\documentclass[a4paper,12pt]{article}
\usepackage[top=2.5cm,bottom=2.5cm,left=2.5cm,right=2.5cm]{geometry}
\usepackage{cite, amsmath, amssymb}
\usepackage[english]{babel}
\usepackage{amsthm}

\usepackage[small,bf,hang]{caption} 

\usepackage{graphicx}
\usepackage{subfig}                      

\usepackage{hyperref}

\usepackage{color}

\newtheorem{theorem}{Theorem}[section]
\theoremstyle{plain}

\newtheorem{lemma}[theorem]{Lemma}

\theoremstyle{definition}

\graphicspath{{figures/}}

\begin{document}


\title{\textbf{Fullerenes with distant pentagons}}

\author{\bigskip\textbf{Jan Goedgebeur$^a$, Brendan D. McKay$^b$}\\
$^a$\textit{Department of Applied Mathematics, Computer Science \& Statistics}\\
\textit{Ghent University}\\
\textit{Krijgslaan 281-S9, 9000 Ghent, Belgium}\\
\medskip
\texttt{jan.goedgebeur@ugent.be}\\
$^b$\textit{Research School of Computer Science}\\  
\textit{Australian National University}\\  
\textit{ACT 2601, Australia}\\
\texttt{bdm@cs.anu.edu.au}
}

\date{} 
\maketitle

\vspace*{-10mm}

\begin{center}
(Received \today)
\end{center}

\begin{abstract}
For each $d>0$, we find all the smallest fullerenes for which the least distance between two pentagons is~$d$.  We also show that for each $d$ there is an $h_d$ such that fullerenes with pentagons at least distance $d$ apart and any number of hexagons greater than or equal to $h_d$ exist.

We also determine the number of fullerenes where the minimum distance between any two pentagons is at least $d$, for $1 \le d \le 5$, up to 400 vertices.

\end{abstract}

\baselineskip=0.30in 


\section{Introduction}
A \textit{fullerene}~\cite{kroto_85} is a cubic plane graph where all faces are pentagons or hexagons. Euler's formula implies that a
fullerene with $n$ vertices contains exactly 12 pentagons and $n/2 - 10$ hexagons.

The
\textit{dual} of a fullerene is the plane graph
obtained by exchanging the roles of vertices and faces: the vertex set of the dual graph
is the set of faces of the original graph and two 
vertices in the dual graph are adjacent if and only if the two faces share an edge
in the original graph. 

The dual of a fullerene with $n$ vertices is a \textit{triangulation} (i.e.\ a plane graph where every face is a triangle) which contains 12 vertices with degree 5 and $n/2 - 10$ vertices with degree 6. The \textit{face-distance} between two pentagons is the distance between the corresponding vertices of degree 5 in the dual graph.

The first fullerene molecule (i.e.\ the $C_{60}$ ``buckyball'') was discovered in 1985 by Kroto et al.~\cite{kroto_85}. Among the fullerenes, the \textit{Isolated Pentagon Rule} (IPR) fullerenes are of special interest as they tend to be more stable~\cite{IPR_ref,IPR_ref2}. IPR  fullerenes  are  fullerenes  where  no  two  pentagons share an edge, i.e.\ they have minimum face-distance at least~2. Raghavachari~\cite{raghavachari1992ground} argued that steric strain will be minimized when the pentagons are distributed as uniformly as possible and therefore proposed the \textit{uniform curvature rule} as an extension of the IPR rule. Also, more recently Rodr\'iguez-Fortea et al.~\cite{rodriguez2010maximum} proposed the maximum pentagon separation rule where they argue that the most suitable carbon cages are those with the largest separations among the 12 pentagons.  These observations lead us to investigate the maximum separation between pentagons that can be achieved for a given number of atoms, or conversely how many atoms are needed to achieve a given separation.
We will refer to the least face-distance between pentagons of a fullerene as the \textit{pentagon separation} of the fullerene.

In the next section we determine the smallest fullerenes with a given pentagon separation. We also show that the minimum fullerenes for each $d$ are unique up to mirror image and that for each $d$ there is an $h_d$ such that fullerenes with pentagon separation at least $d$ and any number of hexagons greater than or equal to $h_d$ exist. The latter was already proven for $h_1$ (i.e., for all fullerenes) by Gr{\"u}nbaum and Motzkin in~\cite{grunbaum1963number} and for $h_2$ (i.e., for IPR fullerenes) by Klein and Liu in~\cite{klein1992theorems}.

Finally, we also determine the number of fullerenes of pentagon separation~$d$, for $1 \le d \le 5$, up to 400 vertices.

\section{Fullerenes with a given minimum pentagon separation}
\label{section:minnv_distant_pentagons}

In this section we determine the smallest fullerenes with a given pentagon separation.

We remind the reader of the icosahedral fullerenes~\cite{goldberg_37,coxeter_71}.
These fullerenes are uniquely determined by their Coxeter coordinates $(p,q)$ and are obtained by cutting an equilateral Goldberg triangle with coordinates $(p,q)$ from the hexagon lattice and gluing it to the faces of the icosahedron. As a Goldberg triangle with coordinates $(p,q)$ has $p^2 + pq + q^2$ vertices, an icosahedral fullerene with Coxeter coordinates $(p,q)$ has $20(p^2 + pq + q^2)$ vertices.  Also note that an icosahedral fullerene with Coxeter coordinates $(p,q)$ has pentagon separation~$p+q$.

The smallest fullerene for $d=1$ is of course unique: the icosahedron $C_{20}$.
For larger~$d$, the minimal fullerenes are given in the next theorem.

\begin{theorem} \label{theorem:min_face_distance_nv}
For odd $d\ge 3$, the smallest fullerenes with pentagon separation
at least $d$ are the icosahedral fullerenes with Coxeter coordinates $(\lceil d/2\rceil,\lfloor d/2\rfloor)$ and $(\lfloor d/2\rfloor,\lceil d/2\rceil)$.  These are mirror images and have $15d^2+5$ vertices.
For even $d$, the unique smallest fullerene with pentagon separation at least $d$ is the
the icosahedral fullerene with Coxeter coordinates $(d/2,d/2)$, which has $15d^2$ vertices.

\end{theorem}

\begin{proof}
\

\noindent\textbf{Proof in the case that $d\ge 3$ is odd:}
\\
The \textit{penta-hexagonal net} is the regular tiling of the plane where a central pentagon is surrounded by an infinite number of hexagons.
The number of faces at face-distance $k$ from the pentagon in the penta-hexagonal net is $5k$. So the number of faces at face-distance at most $k$ from the pentagon in the penta-hexagonal net is $\sum\limits_{i=1}^k 5i + 1 = 5k(k+1)/2 + 1$.
Figure~\ref{fig:d=7patch} shows this situation for $k=3$.

\begin{figure}[h!t]
    \centering
    \includegraphics[width=0.32\textwidth]{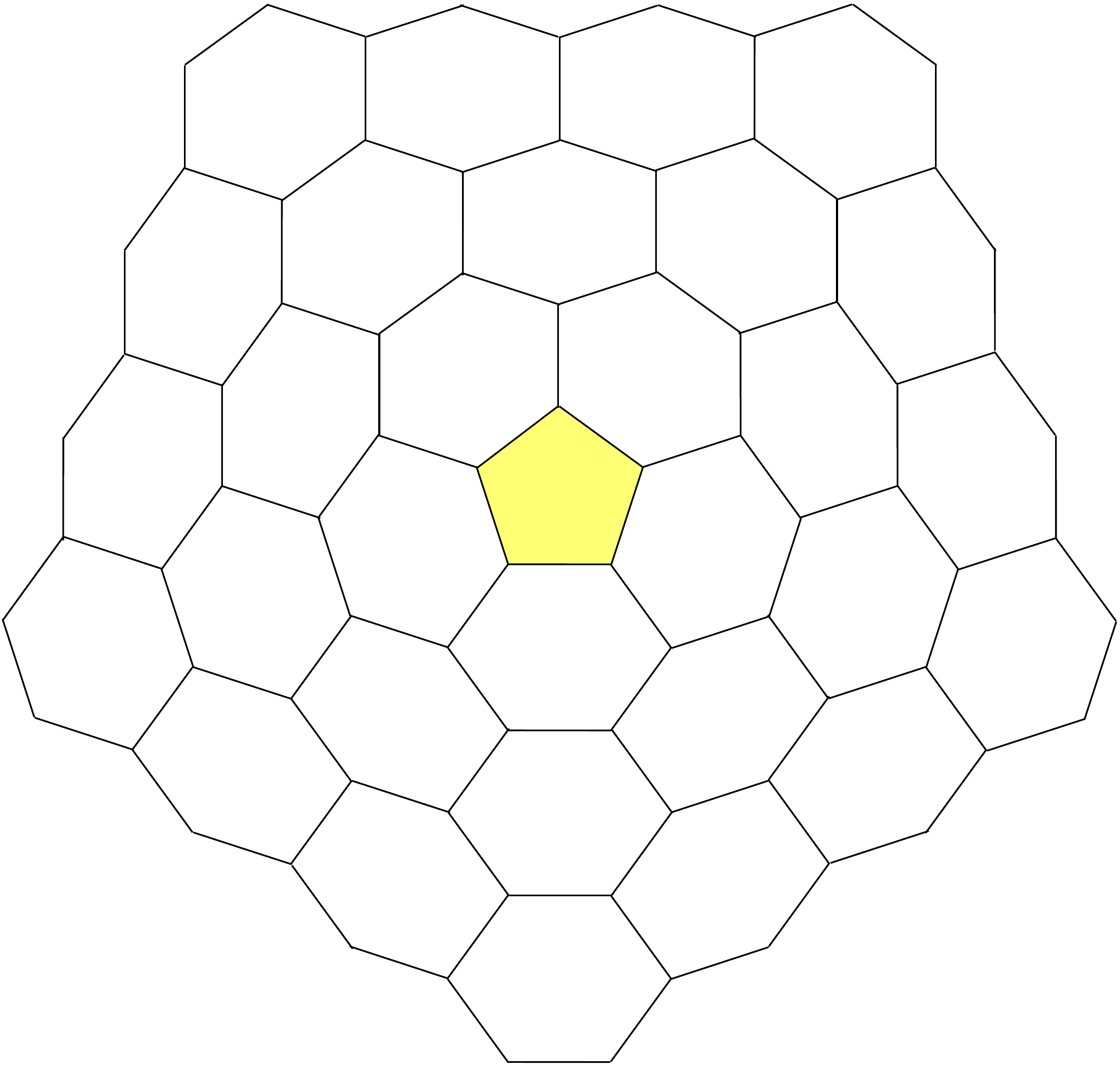}    
    \caption{Patch for $d=7$ in the proof of Theorem~\ref{theorem:min_face_distance_nv}}
    \label{fig:d=7patch}
\end{figure}

In a fullerene with pentagon separation at least~$d$, for odd~$d$, the sets of faces at face-distance at most $\lfloor d/2 \rfloor$ from each pentagon are pairwise disjoint.
Consequently the smallest such fullerenes we can hope to find consist of 12 copies of the above patch for $k=\lfloor d/2 \rfloor$, which comes to $15d^2+5$ vertices.

\begin{figure}[h!t]
    \centering
    \includegraphics[width=0.7\textwidth]{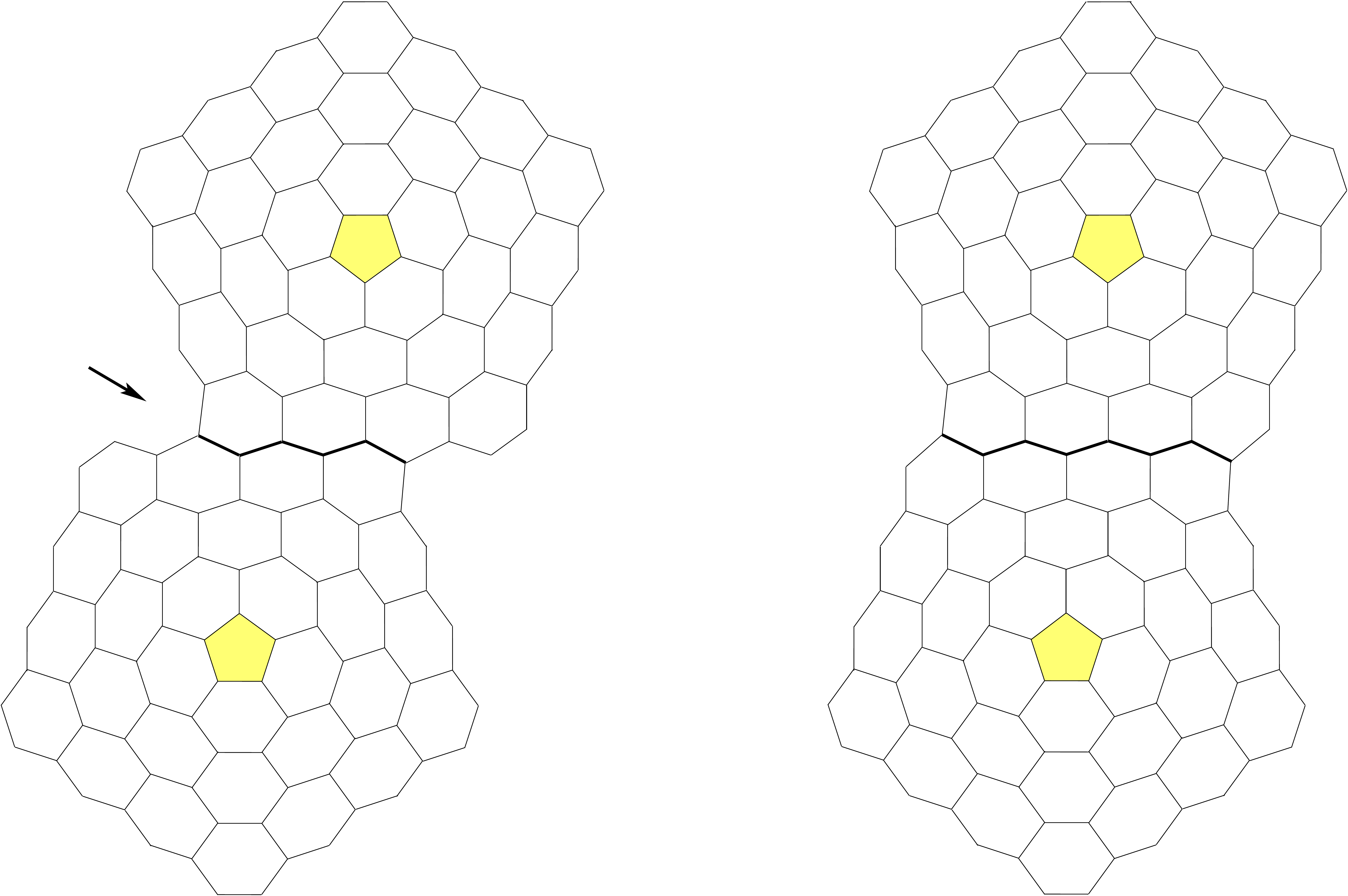}    
    \caption{Bad and good ways to join two patches for $d=7$}
    \label{fig:d=7patchjoin}
\end{figure}

Since the patch boundary has no more than two consecutive vertices of degree~2, it is impossible to join any number of them into a larger patch with a boundary having more than two consecutive vertices of degree~2.  Therefore, considering the complement, no union of these patches which is completable to a fullerene has more than two consecutive vertices of degree~3.  Now, every way to overlap the boundaries of two patches produces three consecutive vertices of degree~3, such as indicated in the left side of Figure~\ref{fig:d=7patchjoin}, except for the way shown in the right side of Figure~\ref{fig:d=7patchjoin} or its mirror image.  For each of these two starting points, there is only one way to attach a third patch to those two patches, and so on, leading to a unique completion in each case.
It is easy to see that these two fullerenes are the icosahedral fullerenes mentioned in the theorem.

\medskip
\begin{figure}[h!t]
    \centering
    \includegraphics[width=0.32\textwidth]{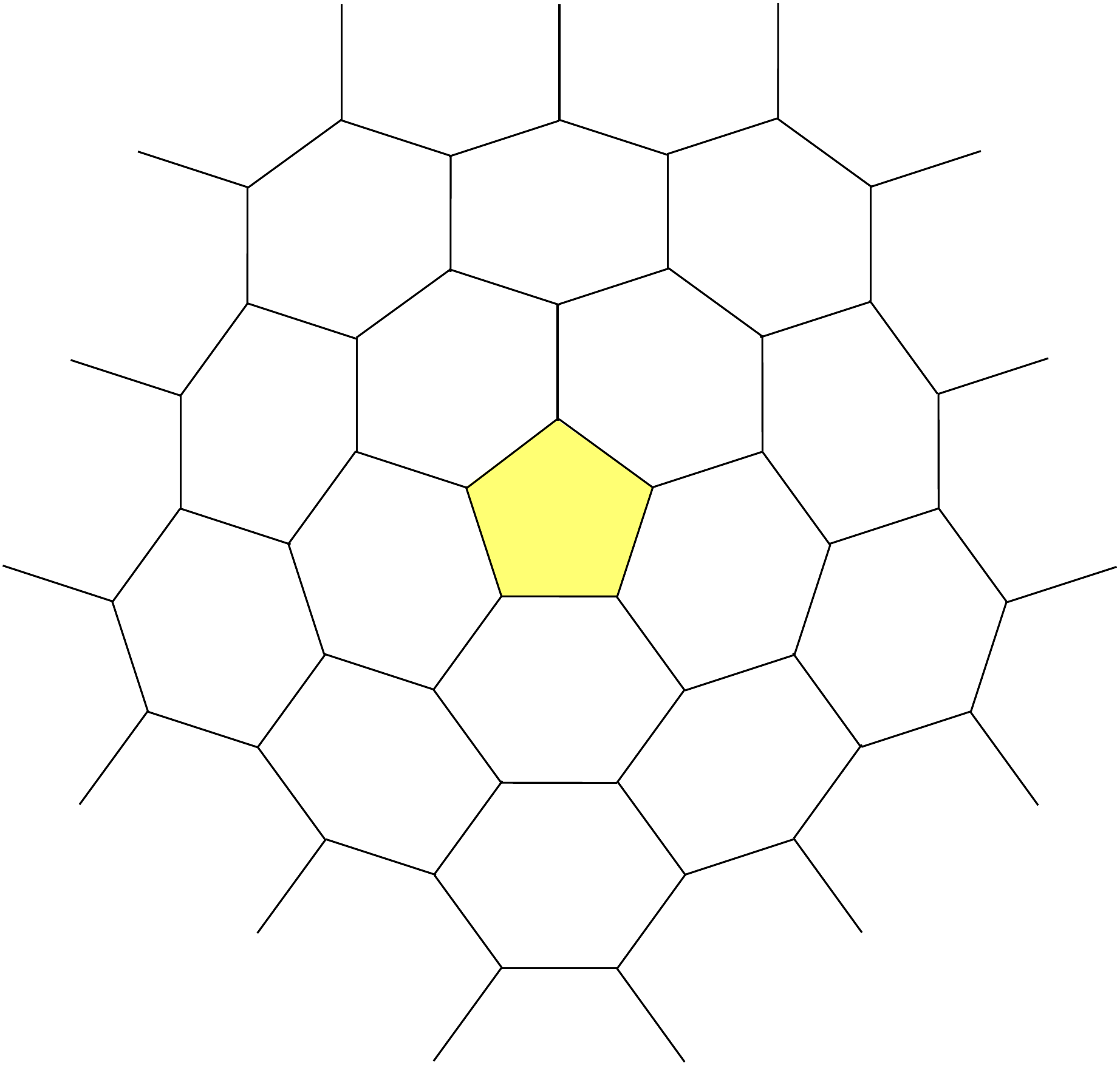}    
    \caption{Patch with dangling edges for $d=6$ in the proof of
              Theorem~\ref{theorem:min_face_distance_nv}}
    \label{fig:d=6patch}
\end{figure}

\noindent\textbf{Proof in the case that $d$ is even:}\\
The proof in this case is similar except that we use a different type of patch.
In~\cite{cvetkovic_02} it was proven that the number of vertices at distance $k$ from the pentagon in the penta-hexagonal net is $5 \lfloor k/2 \rfloor + 5$.  
So the total number of vertices at distance at most $k$ from the pentagon in the penta-hexagonal net is $\sum\limits_{i=0}^k (5 \lfloor i/2 \rfloor + 5) = 5 (\sum\limits_{i=0}^k \lfloor i/2 \rfloor + k + 1)$. If $k$ is even, $\sum\limits_{i=0}^k \lfloor i/2 \rfloor$ is equal to $k^2/4$. So the total number of vertices at distance at most $k$ from the pentagon in the penta-hexagonal net for even $k$ is $5(k^2/4 + k +1)$.

In a fullerene with pentagon separation at least $d$, for even $d$, the sets of vertices at distance at most $d-2$ from every pentagon are pairwise disjoint. 
The case of $d=6$ is shown in Figure~\ref{fig:d=6patch}, excluding the ends of the dangling edges.
Therefore, the smallest such fullerene of pentagon separation $d$ we can hope to construct consists of 12 of these patches for $k=d-2$, joined together by identifying dangling edges.  This would give us $15d^2$ vertices altogether.

\begin{figure}[h!t]
    \centering
    \includegraphics[width=0.7\textwidth]{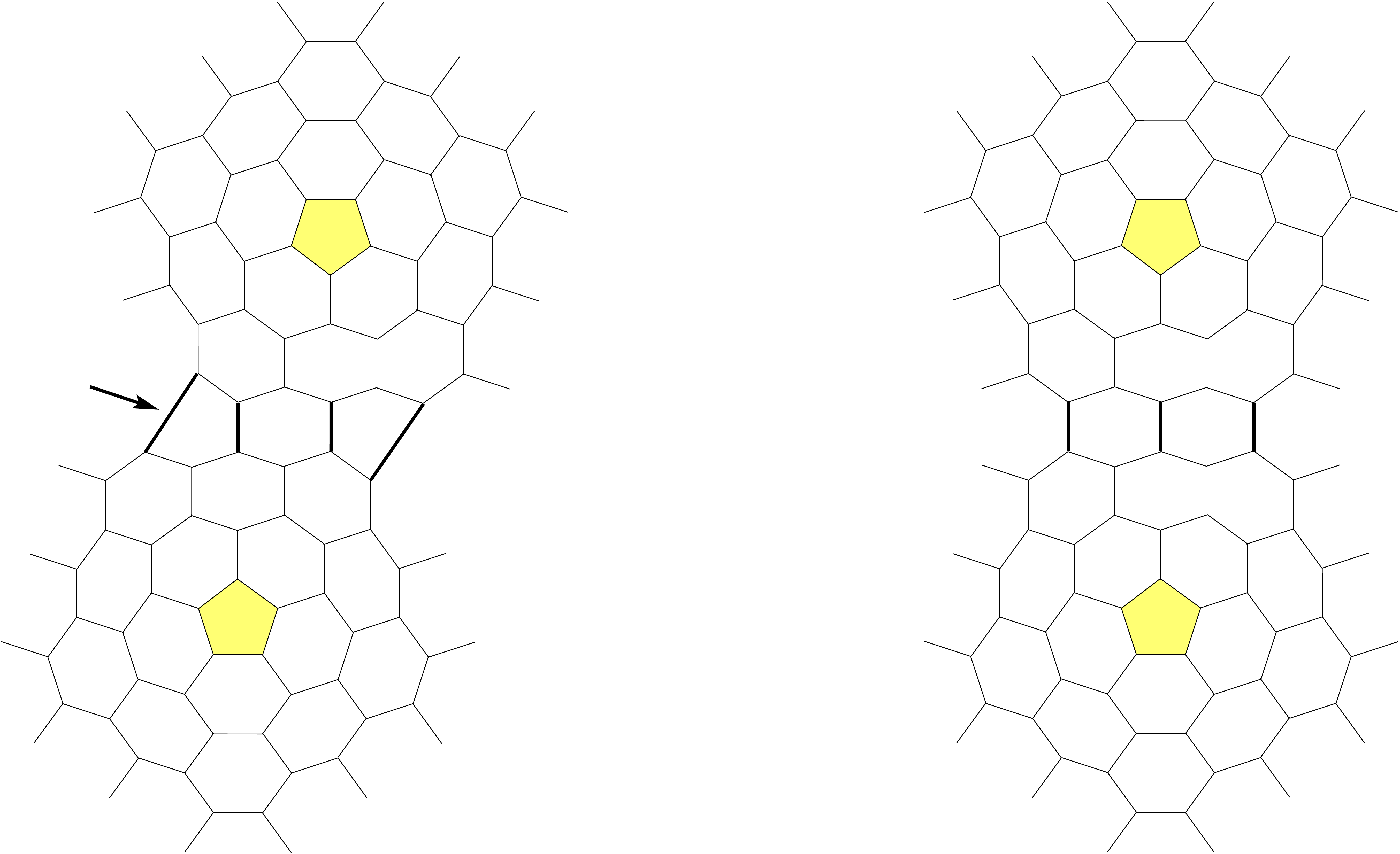}    
    \caption{Bad and good ways to identify dangling edges for $d=6$}
    \label{fig:d=6patchjoin}
\end{figure}

Since we are only permitted to create hexagons incident with the dangling
edges, dangling edges distance two apart in one patch can only be identified with
dangling edges distance two apart in another patch.  Otherwise, a face of the wrong
size is created, such as the pentagon indicated in the left side of Figure~\ref{fig:d=6patchjoin}.
This allows us to join two adjacent patches in only one way, as shown by the
right side of Figure~\ref{fig:d=6patchjoin}. Extra patches can then be attached in
unique fashion, leading to a single fullerene that is easily seen to be the
icosahedral fullerene with Coxeter coordinates $(d/2,d/2)$.
%
%
%
\end{proof}

Next we will prove that for each $d$ there is an $h_d$ such that fullerenes with pentagon separation at least $d$ and any number of hexagons greater than or equal to $h_d$ exist. To prove this, we need Lemmas~\ref{lemma:boundarylength_plus1} and~\ref{lemma:add_hexagons_same_boundary}.

A \textit{fullerene patch} is a connected subgraph of a fullerene where all faces except one exterior face are also faces in the fullerene and all boundary vertices have degree 2 or 3 and all non-boundary vertices have degree 3. The \textit{boundary sequence} of a patch is the cyclic sequence of the degrees of the vertices in the boundary of a patch in clockwise
or counterclockwise order.

A \textit{cap}\index{cap} is a fullerene patch which contains 6 pentagons and has a boundary sequence of the form $(23)^l (32)^m$. Such a boundary is represented by the parameters $(l,m)$. In the literature, the vector $(l,m)$ is also called the \textit{chiral vector} (see~\cite{saito1998physical}). 

\begin{lemma} \label{lemma:boundarylength_plus1}
Any cap with parameters $(l,0)$ can be transformed into a cap with parameters $(l,1)$ without decreasing the minimum face-distance between the pentagons of the cap.
\end{lemma}
\begin{proof}
Given a cap with parameters $(l,0)$. If the cap does not contain a pentagon in its boundary, we remove $(l,0)$ rings of hexagons until there is a pentagon in the boundary of the cap.

In Figure~\ref{fig:change_cap_bound} we show how the $(l,0)$ cap which contains a boundary pentagon (see Figure~\ref{fig:change_cap_bound_step1}) can be transformed into a cap with parameters $(l,1)$ without decreasing the minimum face-distance between the pentagons. This is done by changing the boundary pentagon into a hexagon $h$, adding a ring of hexagons (see Figure~\ref{fig:change_cap_bound_step2}) and changing a hexagon in the boundary which is adjacent to $h$ into a pentagon (see Figure~\ref{fig:change_cap_bound_step3}).
\end{proof}

\begin{figure}[h!t]
    \centering
    \subfloat[]{\label{fig:change_cap_bound_step1}\includegraphics[width=0.7\textwidth]{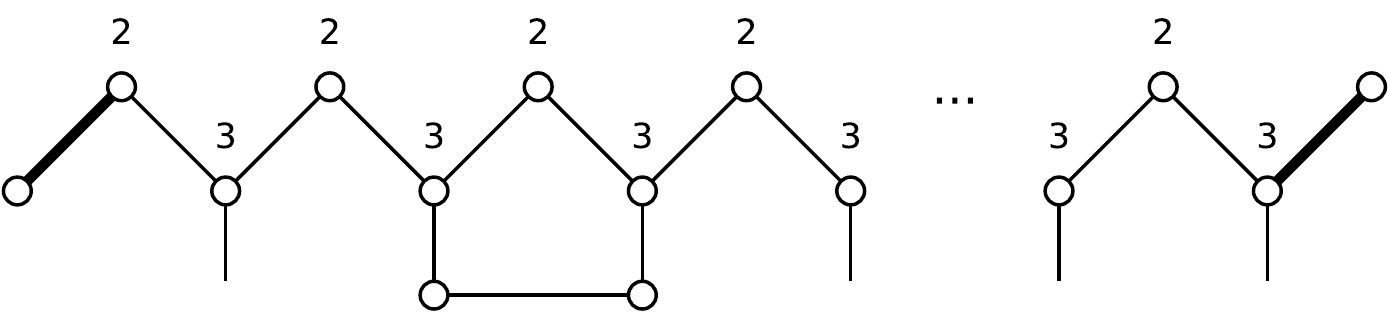}}\\ 
    \subfloat[]{\label{fig:change_cap_bound_step2}\includegraphics[width=0.7\textwidth]{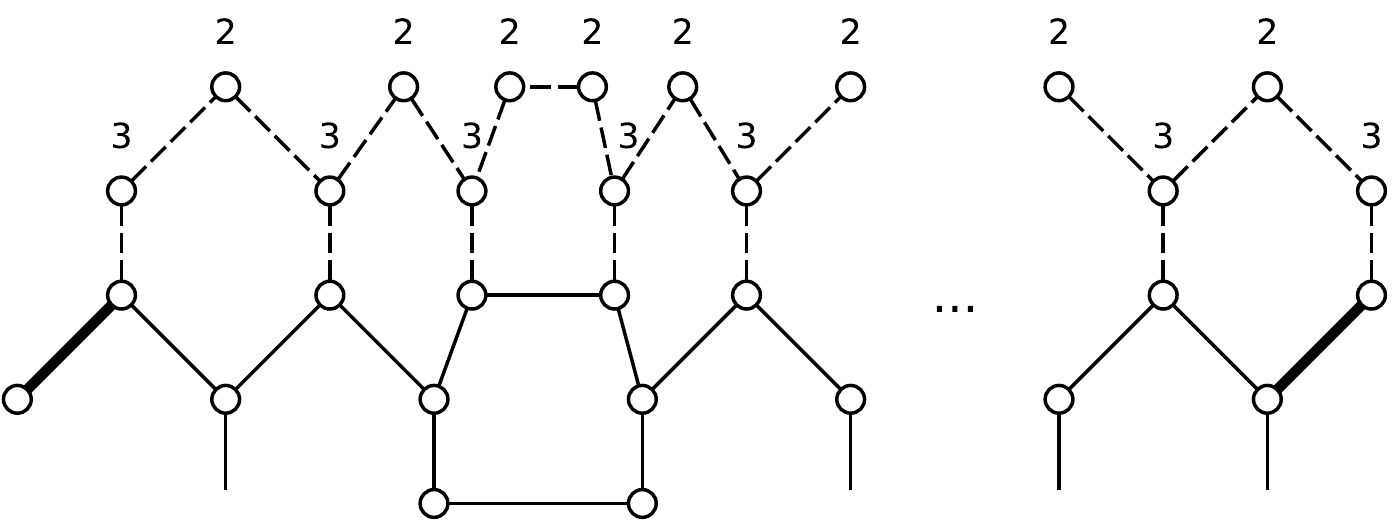}}\\
    \subfloat[]{\label{fig:change_cap_bound_step3}\includegraphics[width=0.7\textwidth]{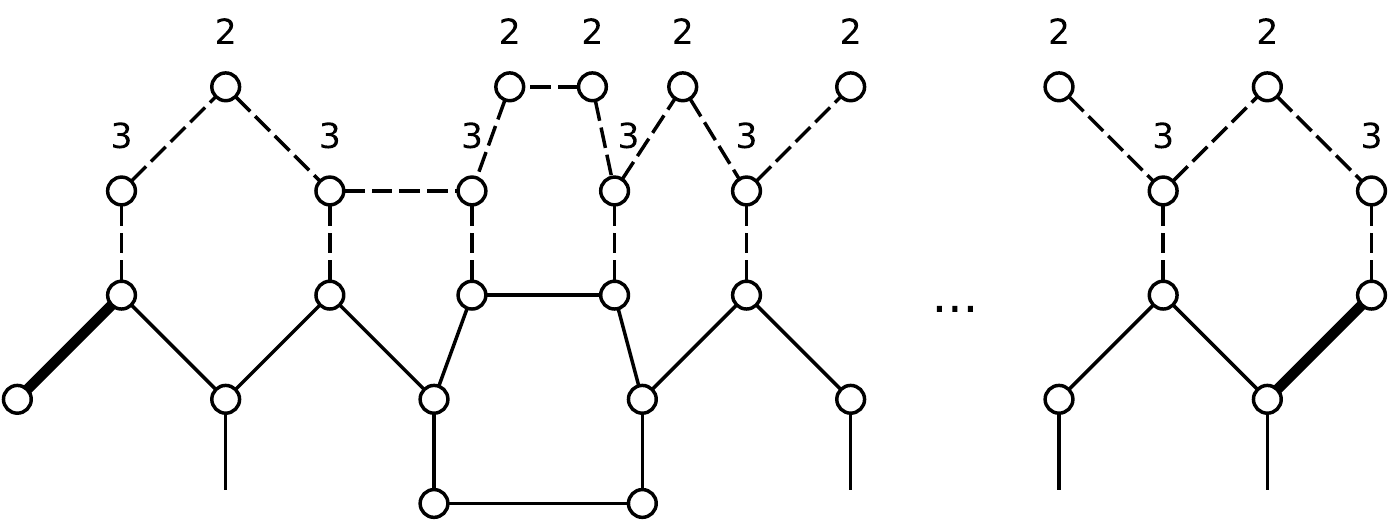}}
    \caption{Procedure to change a cap with parameters $(l,0)$ to a cap with parameters $(l,1)$. The bold edges in the figure have to be identified with each other.}
    \label{fig:change_cap_bound}
\end{figure}

\begin{lemma} \label{lemma:add_hexagons_same_boundary}
Given a cap $C$ with parameters $(l,m)$ with $l \neq 0$ and $m \neq 0$ and which consists of $f$ faces. A cap $C'$ with the same parameters $(l,m)$ which contains $C$ as a subgraph and has $f+l$, respectively $f+m$ faces can be constructed from $C$ by adding $l$ or $m$ hexagons to C, respectively.
\end{lemma}
\begin{proof}
Given a cap $C$ with parameters $(l,m)$ with $l \neq 0$ and $m \neq 0$. In Figure~\ref{fig:add_hexagons_same_cap_step} we show how a cap $C'$ with the same parameters $(l,m)$ which contains $C$ as a subgraph and has $f+l$ faces can be constructed from $C$ by adding $l$ hexagons to C. 

A cap $C''$ with $f+m$ faces can be obtained in a completely analogous way by adding $m$ hexagons to $C$.
\end{proof}

\begin{figure}[h!t]
    \centering
    \subfloat[]{\label{fig:add_hexagons_same_cap_step1}\includegraphics[width=0.95\textwidth]{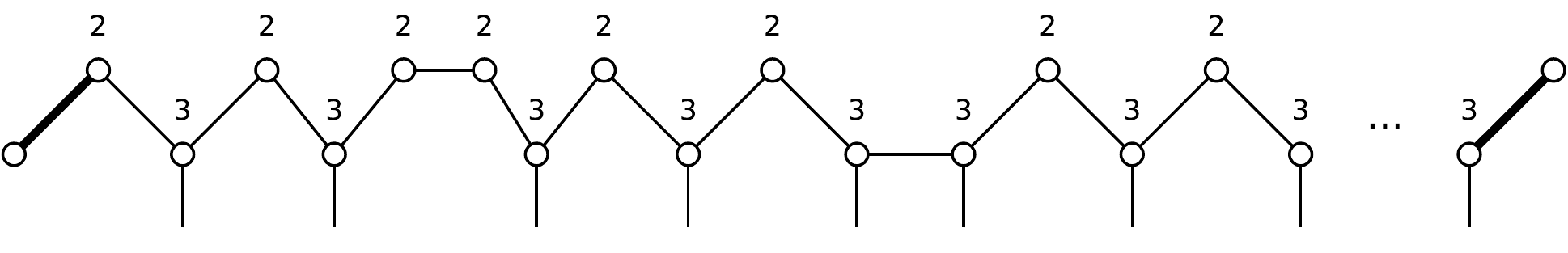}}\\ 
    \subfloat[]{\label{fig:add_hexagons_same_cap_step2}\includegraphics[width=0.95\textwidth]{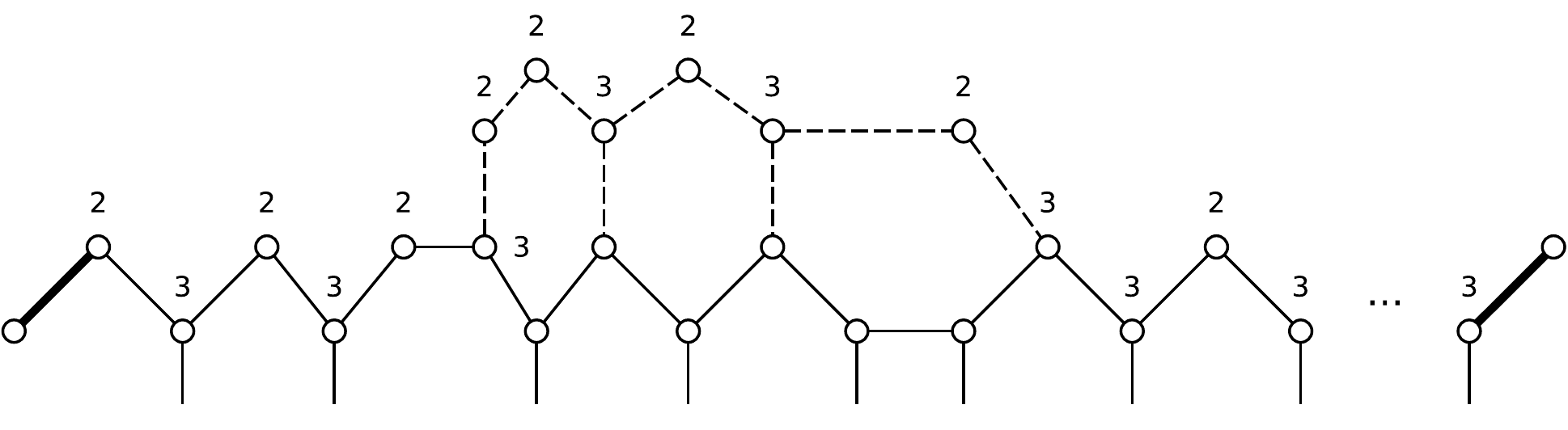}}\\
    \caption{Procedure which adds $l$ hexagons to an $(l,m)$ cap without changing the boundary parameters. The bold edges in the figure have to be identified with each other.}
    \label{fig:add_hexagons_same_cap_step}
\end{figure}

\begin{theorem} \label{theorem:min_face_distance_existence}
For each $d$ there is an $h_d$ such that fullerenes with pentagon separation at least $d$ and any number of hexagons greater than or equal to $h_d$ exist.
\end{theorem}

\begin{proof}
Given an icosahedral fullerene  $F$ with Coxeter coordinates $(\lceil d/2 \rceil, \lceil d/2 \rceil)$. In this fullerene the minimum face-distance between the pentagons is $2\lceil d/2 \rceil$.

Brinkmann and Schein~\cite{brinkmann_schein} have proven that every icosahedral fullerene with Coxeter coordinates $(p,q)$ contains a fullerene patch with 6 pentagons which is a subgraph of a cap with parameters $(3(p+2q),3(p-q))$. So $F$ contains a fullerene patch with 6 pentagons which is a subgraph of a cap with parameters $(9\lceil d/2 \rceil, 0)$.

It follows from~\cite{saito1998physical, ficon_04} that such a fullerene patch can be completed to a cap with parameters $(9\lceil d/2 \rceil, 0)$ by adding hexagons. It follows from Lemma~\ref{lemma:boundarylength_plus1} that this cap can then be transformed to a cap with parameters $(9\lceil d/2 \rceil, 1)$ without decreasing the minimum face-distance between the pentagons of the cap.

We form a fullerene $F'$ with pentagon separation at least $d$ by gluing together two copies of the $(9\lceil d/2 \rceil, 1)$ cap and adding $(9\lceil d/2 \rceil, 1)$ rings of hexagons if necessary. Let $h_{F'}$ denote the number of hexagons of $F'$. Now a fullerene with pentagon separation at least $d$ and any number of hexagons greater than $h_{F'}$ can be obtained by recursively applying Lemma~\ref{lemma:add_hexagons_same_boundary} to $F'$.
\end{proof}

The counts of the number of fullerenes up to 400 vertices with pentagon separation at least $d$, for $1 \le d \le 5$, can be found in Tables~\ref{table:fuller_counts_1}-\ref{table:fuller_counts_4}. (Note that $d=1$ gives the set of all fullerenes and $d=2$ gives the set of all IPR fullerenes). These counts were obtained by using the program \textit{buckygen}~\cite{fuller-paper, fuller-paper-ipr} (which can be downloaded from \url{http://caagt.ugent.be/buckygen/})  to generate all non-isomorphic IPR fullerenes and then applying a separate program to compute their pentagon separation.  Note that fullerenes which are mirror images of each other are considered to be in the same isomorphism class and are thus only counted once.

Some of the fullerenes from Tables~\ref{table:fuller_counts_1}-\ref{table:fuller_counts_4} can also be downloaded from the \textit{House of Graphs}~\cite{hog} at \url{http://hog.grinvin.org/Fullerenes}~.
Figures \ref{fig:smallest_d=3}-\ref{fig:smallest_d=5} show the smallest fullerenes with pentagon separation $d$, for $3 \le d \le 5$.



\begin{figure}[h!t]
    \centering
    \includegraphics[width=0.5\textwidth]{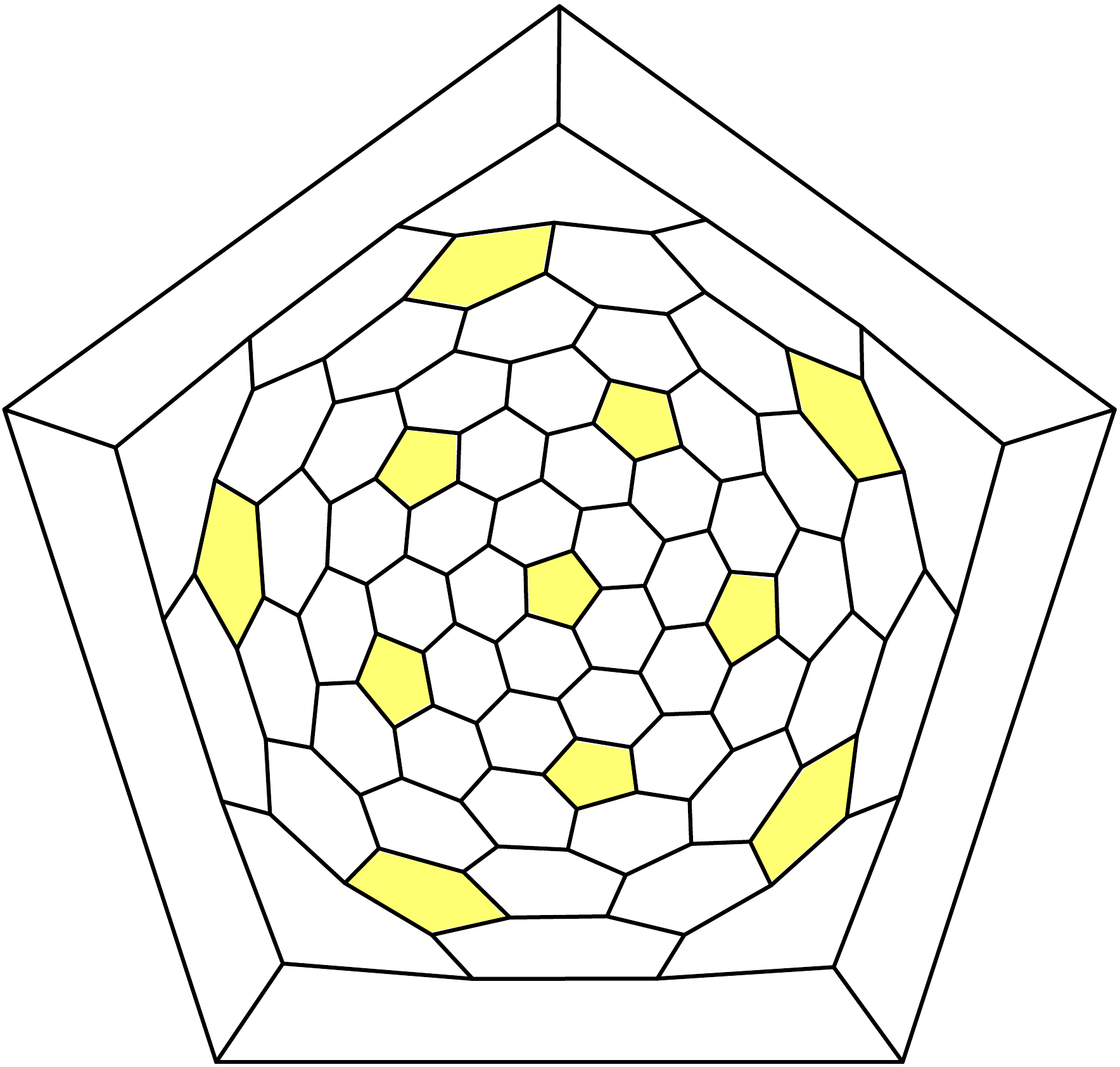}    
    \caption{The icosahedral fullerene with Coxeter coordinates $(2,1)$.  This fullerene and its mirror image are the smallest fullerenes with pentagon separation~3.  They have 140 vertices.}  
    \label{fig:smallest_d=3}
\end{figure}

\begin{figure}[h!t]
    \centering
    \includegraphics[width=0.5\textwidth]{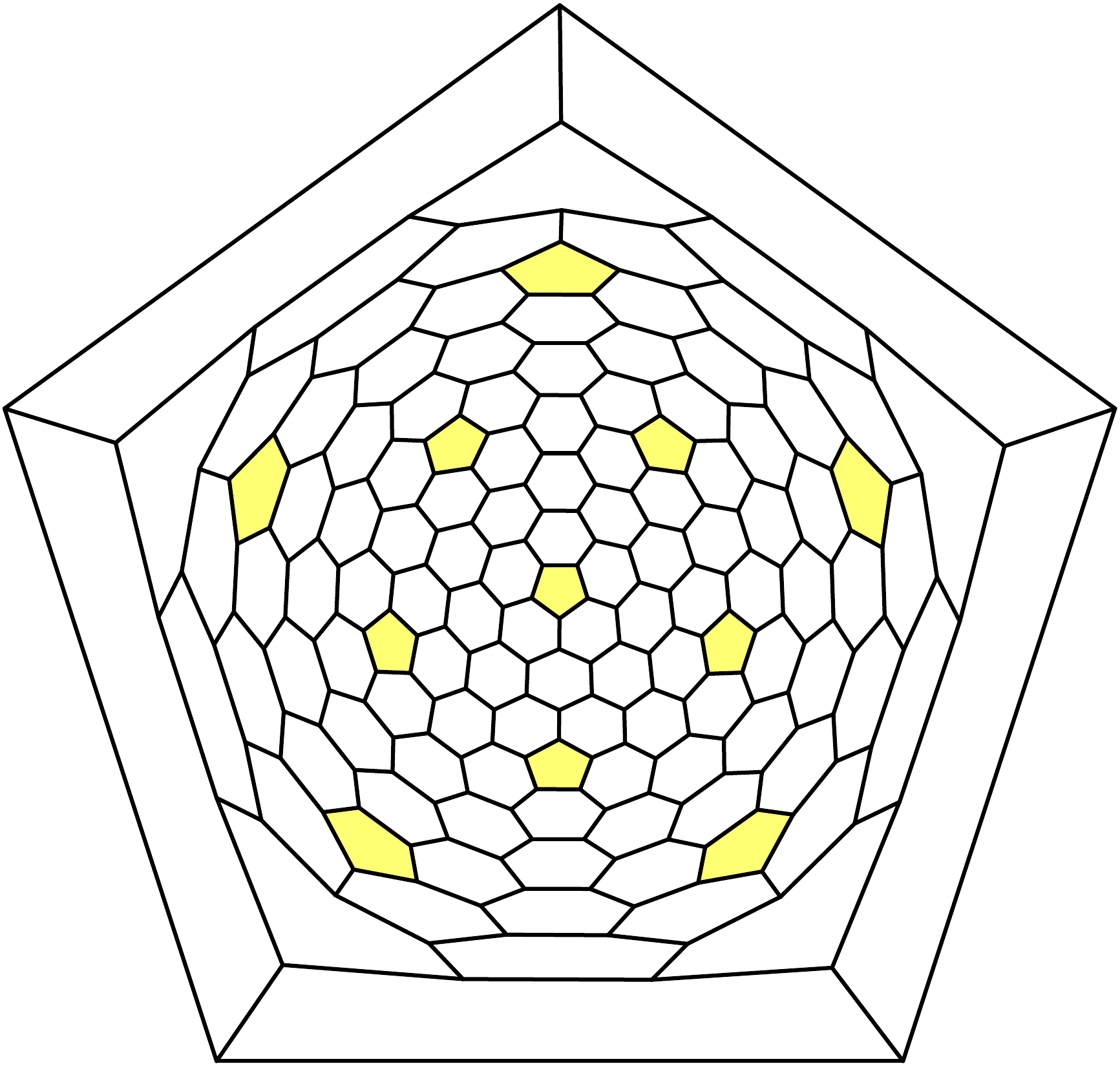}    
    \caption{The icosahedral fullerene with Coxeter coordinates $(2,2)$. This is the smallest fullerene with pentagon separation 4 and has 240 vertices.}
    \label{fig:smallest_d=4}
\end{figure}

\begin{figure}[h!t]
    \centering
    \includegraphics[width=0.5\textwidth]{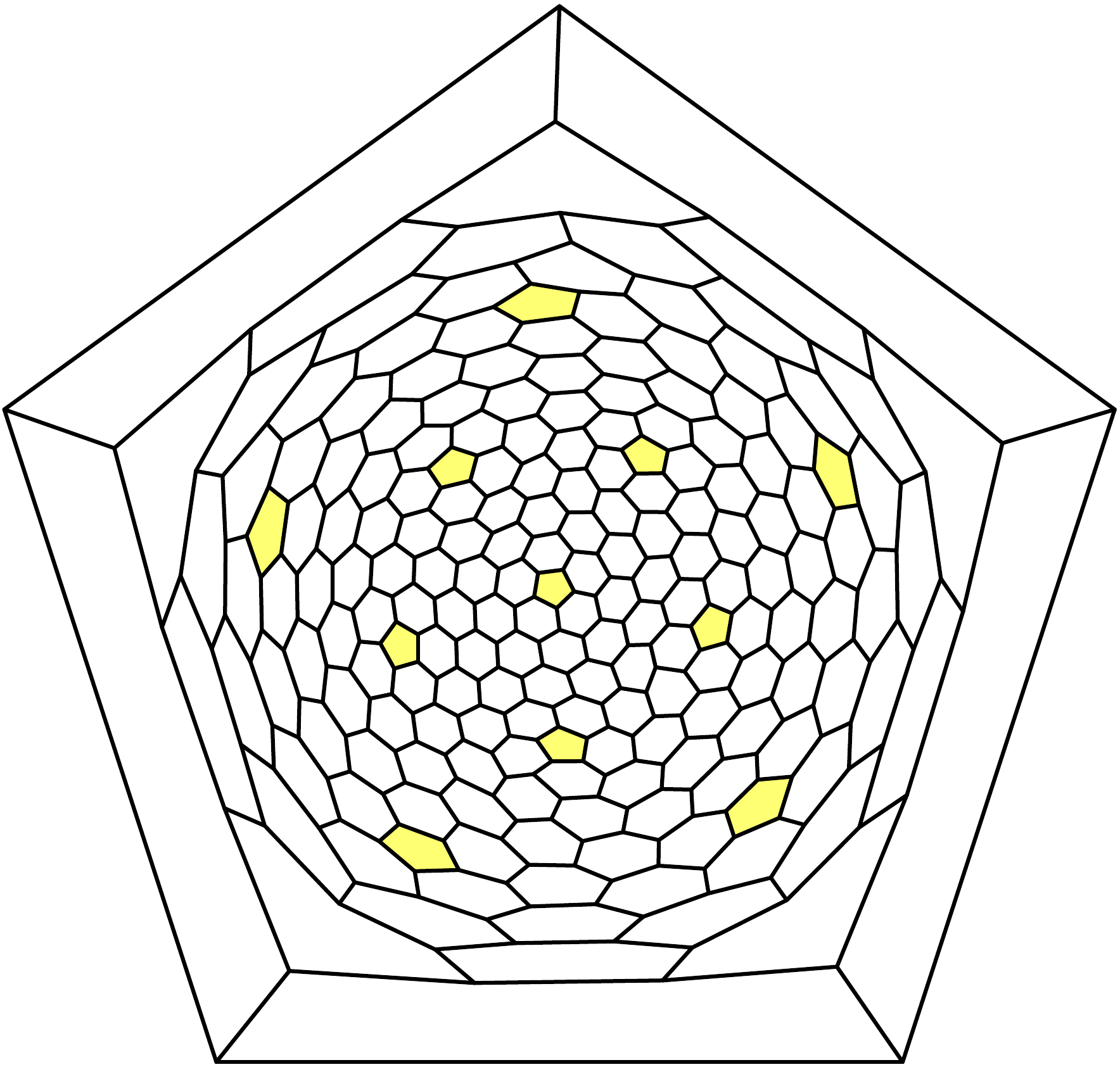}    
    \caption{The icosahedral fullerene with Coxeter coordinates $(3,2)$. This fullerene and its mirror image are the smallest fullerenes with pentagon separation~5.  They have 380 vertices.}  
    \label{fig:smallest_d=5}
\end{figure}

\begin{table}
\centering
{\small 
\begin{tabular}{| c | c | c | c | c | c  | c |}
\hline 
 nv & nf & fullerenes & IPR fullerenes & pent.\, sep.\,${}\ge3$ & pent.\, sep.\,${}\ge4$ & pent.\, sep.\,${}\ge5$\\
\hline 
20  &  12  &  1  &  0  &  0  &  0  &  0\\
22  &  13  &  0  &  0  &  0  &  0  &  0\\
24  &  14  &  1  &  0  &  0  &  0  &  0\\
26  &  15  &  1  &  0  &  0  &  0  &  0\\
28  &  16  &  2  &  0  &  0  &  0  &  0\\
30  &  17  &  3  &  0  &  0  &  0  &  0\\
32  &  18  &  6  &  0  &  0  &  0  &  0\\
34  &  19  &  6  &  0  &  0  &  0  &  0\\
36  &  20  &  15  &  0  &  0  &  0  &  0\\
38  &  21  &  17  &  0  &  0  &  0  &  0\\
40  &  22  &  40  &  0  &  0  &  0  &  0\\
42  &  23  &  45  &  0  &  0  &  0  &  0\\
44  &  24  &  89  &  0  &  0  &  0  &  0\\
46  &  25  &  116  &  0  &  0  &  0  &  0\\
48  &  26  &  199  &  0  &  0  &  0  &  0\\
50  &  27  &  271  &  0  &  0  &  0  &  0\\
52  &  28  &  437  &  0  &  0  &  0  &  0\\
54  &  29  &  580  &  0  &  0  &  0  &  0\\
56  &  30  &  924  &  0  &  0  &  0  &  0\\
58  &  31  &  1 205  &  0  &  0  &  0  &  0\\
60  &  32  &  1 812  &  1  &  0  &  0  &  0\\
62  &  33  &  2 385  &  0  &  0  &  0  &  0\\
64  &  34  &  3 465  &  0  &  0  &  0  &  0\\
66  &  35  &  4 478  &  0  &  0  &  0  &  0\\
68  &  36  &  6 332  &  0  &  0  &  0  &  0\\
70  &  37  &  8 149  &  1  &  0  &  0  &  0\\
72  &  38  &  11 190  &  1  &  0  &  0  &  0\\
74  &  39  &  14 246  &  1  &  0  &  0  &  0\\
76  &  40  &  19 151  &  2  &  0  &  0  &  0\\
78  &  41  &  24 109  &  5  &  0  &  0  &  0\\
80  &  42  &  31 924  &  7  &  0  &  0  &  0\\
82  &  43  &  39 718  &  9  &  0  &  0  &  0\\
84  &  44  &  51 592  &  24  &  0  &  0  &  0\\
86  &  45  &  63 761  &  19  &  0  &  0  &  0\\
88  &  46  &  81 738  &  35  &  0  &  0  &  0\\
90  &  47  &  99 918  &  46  &  0  &  0  &  0\\
92  &  48  &  126 409  &  86  &  0  &  0  &  0\\
94  &  49  &  153 493  &  134  &  0  &  0  &  0\\
96  &  50  &  191 839  &  187  &  0  &  0  &  0\\
98  &  51  &  231 017  &  259  &  0  &  0  &  0\\
100  &  52  &  285 914  &  450  &  0  &  0  &  0\\
102  &  53  &  341 658  &  616  &  0  &  0  &  0\\
104  &  54  &  419 013  &  823  &  0  &  0  &  0\\
106  &  55  &  497 529  &  1 233  &  0  &  0  &  0\\
108  &  56  &  604 217  &  1 799  &  0  &  0  &  0\\
110  &  57  &  713 319  &  2 355  &  0  &  0  &  0\\
112  &  58  &  860 161  &  3 342  &  0  &  0  &  0\\
114  &  59  &  1 008 444  &  4 468  &  0  &  0  &  0\\
\hline
\end{tabular}
}
\caption{Number of fullerenes for a given lower bound on the
pentagon separation. nv is the number of vertices and nf is the number of faces.}

\label{table:fuller_counts_1}

\end{table}

\begin{table}
\centering
{\small 
\begin{tabular}{| c | c | c | c | c | c  | c |}
\hline 
 nv & nf & fullerenes & IPR fullerenes & pent.\, sep.\,${}\ge3$ & pent.\, sep.\,${}\ge4$ & pent.\, sep.\,${}\ge5$\\
\hline 
116  &  60  &  1 207 119  &  6 063  &  0  &  0  &  0\\
118  &  61  &  1 408 553  &  8 148  &  0  &  0  &  0\\
120  &  62  &  1 674 171  &  10 774  &  0  &  0  &  0\\
122  &  63  &  1 942 929  &  13 977  &  0  &  0  &  0\\
124  &  64  &  2 295 721  &  18 769  &  0  &  0  &  0\\
126  &  65  &  2 650 866  &  23 589  &  0  &  0  &  0\\
128  &  66  &  3 114 236  &  30 683  &  0  &  0  &  0\\
130  &  67  &  3 580 637  &  39 393  &  0  &  0  &  0\\
132  &  68  &  4 182 071  &  49 878  &  0  &  0  &  0\\
134  &  69  &  4 787 715  &  62 372  &  0  &  0  &  0\\
136  &  70  &  5 566 949  &  79 362  &  0  &  0  &  0\\
138  &  71  &  6 344 698  &  98 541  &  0  &  0  &  0\\
140  &  72  &  7 341 204  &  121 354  &  1  &  0  &  0\\
142  &  73  &  8 339 033  &  151 201  &  0  &  0  &  0\\
144  &  74  &  9 604 411  &  186 611  &  0  &  0  &  0\\
146  &  75  &  10 867 631  &  225 245  &  0  &  0  &  0\\
148  &  76  &  12 469 092  &  277 930  &  0  &  0  &  0\\
150  &  77  &  14 059 174  &  335 569  &  1  &  0  &  0\\
152  &  78  &  16 066 025  &  404 667  &  2  &  0  &  0\\
154  &  79  &  18 060 979  &  489 646  &  0  &  0  &  0\\
156  &  80  &  20 558 767  &  586 264  &  0  &  0  &  0\\
158  &  81  &  23 037 594  &  697 720  &  0  &  0  &  0\\
160  &  82  &  26 142 839  &  836 497  &  2  &  0  &  0\\
162  &  83  &  29 202 543  &  989 495  &  1  &  0  &  0\\
164  &  84  &  33 022 573  &  1 170 157  &  2  &  0  &  0\\
166  &  85  &  36 798 433  &  1 382 953  &  1  &  0  &  0\\
168  &  86  &  41 478 344  &  1 628 029  &  13  &  0  &  0\\
170  &  87  &  46 088 157  &  1 902 265  &  4  &  0  &  0\\
172  &  88  &  51 809 031  &  2 234 133  &  12  &  0  &  0\\
174  &  89  &  57 417 264  &  2 601 868  &  10  &  0  &  0\\
176  &  90  &  64 353 269  &  3 024 383  &  28  &  0  &  0\\
178  &  91  &  71 163 452  &  3 516 365  &  23  &  0  &  0\\
180  &  92  &  79 538 751  &  4 071 832  &  58  &  0  &  0\\
182  &  93  &  87 738 311  &  4 690 880  &  54  &  0  &  0\\
184  &  94  &  97 841 183  &  5 424 777  &  142  &  0  &  0\\
186  &  95  &  107 679 717  &  6 229 550  &  129  &  0  &  0\\
188  &  96  &  119 761 075  &  7 144 091  &  291  &  0  &  0\\
190  &  97  &  131 561 744  &  8 187 581  &  257  &  0  &  0\\
192  &  98  &  145 976 674  &  9 364 975  &  548  &  0  &  0\\
194  &  99  &  159 999 462  &  10 659 863  &  566  &  0  &  0\\
196  &  100  &  177 175 687  &  12 163 298  &  1 126  &  0  &  0\\
198  &  101  &  193 814 658  &  13 809 901  &  1 072  &  0  &  0\\
200  &  102  &  214 127 742  &  15 655 672  &  1 943  &  0  &  0\\
202  &  103  &  233 846 463  &  17 749 388  &  2 080  &  0  &  0\\
204  &  104  &  257 815 889  &  20 070 486  &  3 682  &  0  &  0\\
206  &  105  &  281 006 325  &  22 606 939  &  3 992  &  0  &  0\\
208  &  106  &  309 273 526  &  25 536 557  &  6 340  &  0  &  0\\
210  &  107  &  336 500 830  &  28 700 677  &  6 737  &  0  &  0\\
\hline
\end{tabular}
}
\caption{Number of fullerenes for a given lower bound on the
pentagon separation (continued). nv is the number of vertices and nf is the number of faces.}

\label{table:fuller_counts_2}

\end{table}

\begin{table}
\centering
{\small 
\begin{tabular}{| c | c | c | c | c | c  | c |}
\hline 
 nv & nf & fullerenes & IPR fullerenes & pent.\, sep.\,${}\ge3$ & pent.\, sep.\,${}\ge4$ & pent.\, sep.\,${}\ge5$\\
\hline 
212  &  108  &  369 580 714  &  32 230 861  &  10 513  &  0  &  0\\
214  &  109  &  401 535 955  &  36 173 081  &  12 000  &  0  &  0\\
216  &  110  &  440 216 206  &  40 536 922  &  18 169  &  0  &  0\\
218  &  111  &  477 420 176  &  45 278 722  &  20 019  &  0  &  0\\
220  &  112  &  522 599 564  &  50 651 799  &  28 528  &  0  &  0\\
222  &  113  &  565 900 181  &  56 463 948  &  32 276  &  0  &  0\\
224  &  114  &  618 309 598  &  62 887 775  &  46 534  &  0  &  0\\
226  &  115  &  668 662 698  &  69 995 887  &  52 177  &  0  &  0\\
228  &  116  &  729 414 880  &  77 831 323  &  71 303  &  0  &  0\\
230  &  117  &  787 556 069  &  86 238 206  &  79 915  &  0  &  0\\
232  &  118  &  857 934 016  &  95 758 929  &  109 848  &  0  &  0\\
234  &  119  &  925 042 498  &  105 965 373  &  124 153  &  0  &  0\\
236  &  120  &  1 006 016 526  &  117 166 528  &  164 700  &  0  &  0\\
238  &  121  &  1 083 451 816  &  129 476 607  &  184 404  &  0  &  0\\
240  &  122  &  1 176 632 247  &  142 960 479  &  242 507  &  1  &  0\\
242  &  123  &  1 265 323 971  &  157 402 781  &  273 885  &  0  &  0\\
244  &  124  &  1 372 440 782  &  173 577 766  &  353 997  &  0  &  0\\
246  &  125  &  1 474 111 053  &  190 809 628  &  397 673  &  0  &  0\\
248  &  126  &  1 596 482 232  &  209 715 141  &  507 913  &  0  &  0\\
250  &  127  &  1 712 934 069  &  230 272 559  &  570 053  &  0  &  0\\
252  &  128  &  1 852 762 875  &  252 745 513  &  717 983  &  0  &  0\\
254  &  129  &  1 985 250 572  &  276 599 787  &  805 374  &  0  &  0\\
256  &  130  &  2 144 943 655  &  303 235 792  &  1 007 680  &  0  &  0\\
258  &  131  &  2 295 793 276  &  331 516 984  &  1 127 989  &  0  &  0\\
260  &  132  &  2 477 017 558  &  362 302 637  &  1 392 996  &  2  &  0\\
262  &  133  &  2 648 697 036  &  395 600 325  &  1 550 580  &  0  &  0\\
264  &  134  &  2 854 536 850  &  431 894 257  &  1 905 849  &  0  &  0\\
266  &  135  &  3 048 609 900  &  470 256 444  &  2 124 873  &  1  &  0\\
268  &  136  &  3 282 202 941  &  512 858 451  &  2 592 104  &  1  &  0\\
270  &  137  &  3 501 931 260  &  557 745 670  &  2 868 467  &  2  &  0\\
272  &  138  &  3 765 465 341  &  606 668 511  &  3 461 487  &  1  &  0\\
274  &  139  &  4 014 007 928  &  659 140 287  &  3 847 594  &  0  &  0\\
276  &  140  &  4 311 652 376  &  716 217 922  &  4 621 524  &  1  &  0\\
278  &  141  &  4 591 045 471  &  776 165 188  &  5 112 067  &  2  &  0\\
280  &  142  &  4 926 987 377  &  842 498 881  &  6 079 570  &  4  &  0\\
282  &  143  &  5 241 548 270  &  912 274 540  &  6 726 996  &  1  &  0\\
284  &  144  &  5 618 445 787  &  987 874 095  &  7 971 111  &  10  &  0\\
286  &  145  &  5 972 426 835  &  1 068 507 788  &  8 784 514  &  3  &  0\\
288  &  146  &  6 395 981 131  &  1 156 161 307  &  10 352 546  &  7  &  0\\
290  &  147  &  6 791 769 082  &  1 247 686 189  &  11 385 724  &  9  &  0\\
292  &  148  &  7 267 283 603  &  1 348 832 364  &  13 357 318  &  5  &  0\\
294  &  149  &  7 710 782 991  &  1 454 359 806  &  14 652 198  &  6  &  0\\
296  &  150  &  8 241 719 706  &  1 568 768 524  &  17 102 231  &  24  &  0\\
298  &  151  &  8 738 236 515  &  1 690 214 836  &  18 756 139  &  16  &  0\\
300  &  152  &  9 332 065 811  &  1 821 766 896  &  21 766 152  &  32  &  0\\
302  &  153  &  9 884 604 767  &  1 958 581 588  &  23 815 310  &  36  &  0\\
304  &  154  &  10 548 218 751  &  2 109 271 290  &  27 529 516  &  46  &  0\\
306  &  155  &  11 164 542 762  &  2 266 138 871  &  30 090 574  &  54  &  0\\
\hline
\end{tabular}
}
\caption{Number of fullerenes for a given lower bound on the
pentagon separation (continued). nv is the number of vertices and nf is the number of faces.}

\label{table:fuller_counts_3}

\end{table}

\begin{table}
\centering
{\small 
\begin{tabular}{| c | c | c | c | c | c  | c |}
\hline 
 nv & nf & fullerenes & IPR fullerenes & pent.\, sep.\,${}\ge3$ & pent.\, sep.\,${}\ge4$ & pent.\, sep.\,${}\ge5$\\
\hline 
308  &  156  &  11 902 015 724  &  2 435 848 971  &  34 629 672  &  99  &  0\\
310  &  157  &  12 588 998 862  &  2 614 544 391  &  37 770 691  &  93  &  0\\
312  &  158  &  13 410 330 482  &  2 808 510 141  &  43 312 313  &  135  &  0\\
314  &  159  &  14 171 344 797  &  3 009 120 113  &  47 153 778  &  187  &  0\\
316  &  160  &  15 085 164 571  &  3 229 731 630  &  53 899 686  &  211  &  0\\
318  &  161  &  15 930 619 304  &  3 458 148 016  &  58 585 441  &  308  &  0\\
320  &  162  &  16 942 010 457  &  3 704 939 275  &  66 712 070  &  443  &  0\\
322  &  163  &  17 880 232 383  &  3 964 153 268  &  72 395 888  &  535  &  0\\
324  &  164  &  19 002 055 537  &  4 244 706 701  &  82 171 212  &  698  &  0\\
326  &  165  &  20 037 346 408  &  4 533 465 777  &  89 063 353  &  1 026  &  0\\
328  &  166  &  21 280 571 390  &  4 850 870 260  &  100 785 130  &  1 216  &  0\\
330  &  167  &  22 426 253 115  &  5 178 120 469  &  109 068 073  &  1 623  &  0\\
332  &  168  &  23 796 620 378  &  5 531 727 283  &  122 992 213  &  2 489  &  0\\
334  &  169  &  25 063 227 406  &  5 900 369 830  &  132 950 223  &  2 788  &  0\\
336  &  170  &  26 577 912 084  &  6 299 880 577  &  149 523 121  &  3 612  &  0\\
338  &  171  &  27 970 034 826  &  6 709 574 675  &  161 430 830  &  4 744  &  0\\
340  &  172  &  29 642 262 229  &  7 158 963 073  &  181 076 418  &  5 845  &  0\\
342  &  173  &  31 177 474 996  &  7 620 446 934  &  195 124 334  &  7 457  &  0\\
344  &  174  &  33 014 225 318  &  8 118 481 242  &  218 323 289  &  10 591  &  0\\
346  &  175  &  34 705 254 287  &  8 636 262 789  &  235 050 400  &  12 307  &  0\\
348  &  176  &  36 728 266 430  &  9 196 920 285  &  262 381 050  &  15 312  &  0\\
350  &  177  &  38 580 626 759  &  9 768 511 147  &  282 042 413  &  19 574  &  0\\
352  &  178  &  40 806 395 661  &  10 396 040 696  &  314 052 518  &  23 755  &  0\\
354  &  179  &  42 842 199 753  &  11 037 658 075  &  337 229 970  &  29 793  &  0\\
356  &  180  &  45 278 616 586  &  11 730 538 496  &  374 666 300  &  38 688  &  0\\
358  &  181  &  47 513 679 057  &  12 446 446 419  &  401 932 458  &  45 946  &  0\\
360  &  182  &  50 189 039 868  &  13 221 751 502  &  445 482 235  &  55 742  &  0\\
362  &  183  &  52 628 839 448  &  14 010 515 381  &  477 264 068  &  69 970  &  0\\
364  &  184  &  55 562 506 886  &  14 874 753 568  &  528 016 753  &  83 616  &  0\\
366  &  185  &  58 236 270 451  &  15 754 940 959  &  565 045 586  &  100 644  &  0\\
368  &  186  &  61 437 700 788  &  16 705 334 454  &  623 895 236  &  126 048  &  0\\
370  &  187  &  64 363 670 678  &  17 683 643 273  &  666 935 811  &  149 044  &  0\\
372  &  188  &  67 868 149 215  &  18 744 292 915  &  734 907 336  &  179 013  &  0\\
374  &  189  &  71 052 718 441  &  19 816 289 281  &  784 797 263  &  217 673  &  0\\
376  &  190  &  74 884 539 987  &  20 992 425 825  &  863 237 405  &  257 673  &  0\\
378  &  191  &  78 364 039 771  &  22 186 413 139  &  920 935 351  &  302 553  &  0\\
380  &  192  &  82 532 990 559  &  23 475 079 272  &  1 011 152 383  &  367 547  &  1\\
382  &  193  &  86 329 680 991  &  24 795 898 388  &  1 077 679 749  &  434 339  &  0\\
384  &  194  &  90 881 152 117  &  26 227 197 453  &  1 181 149 036  &  507 481  &  0\\
386  &  195  &  95 001 297 565  &  27 670 862 550  &  1 257 630 423  &  611 532  &  0\\
388  &  196  &  99 963 147 805  &  29 254 036 711  &  1 376 400 812  &  707 184  &  0\\
390  &  197  &  104 453 597 992  &  30 852 950 986  &  1 463 926 563  &  820 525  &  0\\
392  &  198  &  109 837 310 021  &  32 581 366 295  &  1 599 524 989  &  982 532  &  0\\
394  &  199  &  114 722 988 623  &  34 345 173 894  &  1 699 970 613  &  1 133 377  &  0\\
396  &  200  &  120 585 261 143  &  36 259 212 641  &  1 854 374 011  &  1 323 509  &  0\\
398  &  201  &  125 873 325 588  &  38 179 777 473  &  1 969 147 856  &  1 546 304  &  0\\
400  &  202  &  132 247 999 328  &  40 286 153 024  &  2 144 985 583  &  1 784 313  &  1\\
\hline
\end{tabular}
}
\caption{Number of fullerenes for a given lower bound on the
pentagon separation (continued). nv is the number of vertices and nf is the number of faces.}

\label{table:fuller_counts_4}

\end{table}

\begin{flushleft}
\textit{Acknowledgements:}
Jan Goedgebeur is supported by a Postdoctoral Fellowship of the Research Foundation Flanders (FWO).  Brendan McKay is supported by the Australian Research Council. Most computations for this work were carried out using the Stevin Supercomputer Infrastructure at Ghent University. We also would like to thank Gunnar Brinkmann, Patrick Fowler and Jack Graver for useful suggestions.
\end{flushleft}

\bibliographystyle{plain}


\end{document}